\documentclass[dvips]{amsart}

\usepackage{mathtools, amsmath, amssymb, amsthm, url, color, enumerate}
\usepackage{graphicx}
\usepackage{comment} 
\usepackage{wrapfig} 
\usepackage{multicol}
\usepackage[varg]{txfonts}
\definecolor{ikkonzome}{rgb}	{	0.9961	,	0.7569	,	0.7373	}
\definecolor{ishitake}{rgb}	{	0.9961	,	0.6941	,	0.7059	}
\definecolor{momo}{rgb}	        {	0.9961	,	0.6824	,	0.8039	}
\definecolor{kobai}{rgb}	{	0.9412	,	0.4235	,	0.5569	}
\definecolor{nakabeni}{rgb}	{	0.9451	,	0.2706	,	0.4941	}
\definecolor{sakura}{rgb}	{	0.9333	,	0.8353	,	0.8353	}
\definecolor{arazome}{rgb}	{	0.9725	,	0.7216	,	0.7843	}
\definecolor{usubeni}{rgb}	{	0.8471	,	0.4471	,	0.5451	}
\definecolor{hisame}{rgb}	{	0.7451	,	0.4039	,	0.4039	}
\definecolor{toki}{rgb}	        {	0.9569	,	0.6431	,	0.6353	}
\definecolor{sakuranezumi}{rgb}	{	0.6941	,	0.6039	,	0.6078	}
\definecolor{sango}	{rgb}	{	0.8471	,	0.4157	,	0.3725	}
\definecolor{akane}	{rgb}	{	0.7529	,	0.0118	,	0.3451	}
\definecolor{choshun}{rgb}	{	0.7490	,	0.5255	,	0.5255	}
\definecolor{karakurenai}{rgb}	{	0.7373	,	0.0118	,	0.2667	}
\definecolor{enji}{rgb}	        {	0.6275	,	0.0863	,	0.3176	}
\definecolor{keshiaka}{rgb}	{	0.6275	,	0.4275	,	0.4275	}
\definecolor{kokiake}{rgb}	{	0.5059	,	0.0706	,	0.2549	}
\definecolor{jinzamomi}{rgb}	{	0.9098	,	0.4549	,	0.4157	}
\definecolor{mizugaki}{rgb}	{	0.7294	,	0.5529	,	0.4784	}
\definecolor{umenezumi}{rgb}	{	0.5882	,	0.3922	,	0.3882	}
\definecolor{suoko}{rgb}        {	0.5843	,	0.2667	,	0.2431	}
\definecolor{akabeni}{rgb}	{	0.8039	,	0.0784	,	0.3725	}
\definecolor{shinshu}{rgb}	{	0.6431	,	0.0353	,	0.0000	}
\definecolor{azuki}{rgb}	{	0.5255	,	0.0235	,	0.0000	}
\definecolor{ginshu}{rgb}	{	0.7490	,	0.2824	,	0.0588	}
\definecolor{ebicha}{rgb}	{	0.4549	,	0.2706	,	0.2627	}
\definecolor{kuriume}{rgb}	{	0.5843	,	0.3804	,	0.4314	}
\definecolor{akebono}{rgb}	{	0.8902	,	0.5961	,	0.4941	}
\definecolor{hanezu}{rgb}	{	0.7882	,	0.5961	,	0.5373	}
\definecolor{sangoshu}{rgb}	{	0.8196	,	0.5059	,	0.4471	}
\definecolor{shozyohi}{rgb}	{	0.7686	,	0.0000	,	0.0000	}
\definecolor{shikancha}{rgb}	{	0.5569	,	0.3294	,	0.1882	}
\definecolor{kakishibu}{rgb}	{	0.6745	,	0.4078	,	0.3333	}
\definecolor{benikaba}{rgb}	{	0.7137	,	0.3373	,	0.2941	}
\definecolor{benitobi}{rgb}	{	0.6196	,	0.3176	,	0.2706	}
\definecolor{benihihada}{rgb}	{	0.5020	,	0.3137	,	0.2353	}
\definecolor{kurotobi}{rgb}	{	0.3176	,	0.2000	,	0.1490	}
\definecolor{benihi}{rgb}	{	0.8235	,	0.4745	,	0.1922	}
\definecolor{terigaki}{rgb}	{	0.8118	,	0.4627	,	0.1804	}
\definecolor{ake}{rgb}	        {	0.7804	,	0.3098	,	0.1725	}
\definecolor{edocha}{rgb}	{	0.6863	,	0.4353	,	0.2941	}
\definecolor{bengara}{rgb}	{	0.6392	,	0.1569	,	0.0196	}
\definecolor{hihada}{rgb}	{	0.5412	,	0.3412	,	0.2353	}
\definecolor{shishi}{rgb}	{	0.8549	,	0.6863	,	0.5961	}
\definecolor{araishu}{rgb}	{	0.9294	,	0.4902	,	0.4549	}
\definecolor{akago}{rgb}	{	0.8118	,	0.5765	,	0.4275	}
\definecolor{tokigaracha}{rgb}	{	0.7922	,	0.5255	,	0.3686	}
\definecolor{otan}{rgb}	        {	0.8157	,	0.4157	,	0.2235	}
\definecolor{komugi}{rgb}	{	0.8157	,	0.6549	,	0.5098	}
\definecolor{rakuda}{rgb}	{	0.6784	,	0.5255	,	0.4118	}
\definecolor{tsurubami}{rgb}	{	0.6275	,	0.4392	,	0.3961	}
\definecolor{ama}{rgb}	        {	0.7765	,	0.6902	,	0.5843	}
\definecolor{nikkei}{rgb}	{	0.7216	,	0.4667	,	0.3725	}
\definecolor{renga}{rgb}	{	0.6902	,	0.3765	,	0.3098	}
\definecolor{sohi}{rgb}   	{	0.8078	,	0.5098	,	0.2078	}
\definecolor{enshucha}{rgb}	{	0.6706	,	0.4275	,	0.1608	}
\definecolor{karacha}{rgb}	{	0.5765	,	0.4235	,	0.1490	}
\definecolor{kabacha}{rgb}	{	0.6353	,	0.3725	,	0.1569	}
\definecolor{sodenkaracha}{rgb}	{	0.5216	,	0.3490	,	0.1373	}
\definecolor{suzumecha}{rgb}	{	0.4745	,	0.3176	,	0.1255	}
\definecolor{kurikawacha}{rgb}	{	0.4078	,	0.2745	,	0.1098	}
\definecolor{momoshiocha}{rgb}	{	0.3490	,	0.2353	,	0.0902	}
\definecolor{tobi}{rgb}	        {	0.4353	,	0.3098	,	0.1412	}
\definecolor{kurumizome}{rgb}	{	0.6667	,	0.5333	,	0.3333	}
\definecolor{kaba}{rgb}	        {	0.8196	,	0.4588	,	0.1294	}
\definecolor{korosen}{rgb}	{	0.5059	,	0.3843	,	0.1608	}
\definecolor{kogecha}{rgb}	{	0.3333	,	0.2549	,	0.1098	}
\definecolor{kokikuchinashi}{rgb}	{	0.8078	,	0.5922	,	0.3490	}
\definecolor{araigaki}{rgb}	{	0.8157	,	0.5529	,	0.3176	}
\definecolor{taisha}{rgb}	{	0.6353	,	0.4196	,	0.2078	}
\definecolor{akashirotsurubami}{rgb}	{	0.8078	,	0.6118	,	0.4157	}
\definecolor{tonocha}{rgb}	{	0.5922	,	0.4039	,	0.2000	}
\definecolor{sencha}{rgb}	{	0.5255	,	0.3529	,	0.1490	}
\definecolor{sharegaki}{rgb}	{	0.8549	,	0.7098	,	0.4863	}
\definecolor{ko}{rgb}	        {	0.9529	,	0.8275	,	0.6627	}
\definecolor{usugaki}{rgb}	{	0.8627	,	0.7294	,	0.5294	}
\definecolor{koji}{rgb}	        {	0.8275	,	0.6039	,	0.2706	}
\definecolor{umezome}{rgb}	{	0.8667	,	0.7373	,	0.4235	}
\definecolor{beniukon}{rgb}	{	0.8549	,	0.5843	,	0.2549	}
\definecolor{chojicha}{rgb}	{	0.5569	,	0.4000	,	0.1098	}
\definecolor{kenpozome}{rgb}	{	0.3059	,	0.2667	,	0.0627	}
\definecolor{biwacha}{rgb}	{	0.7333	,	0.5294	,	0.1922	}
\definecolor{kohaku}{rgb}	{	0.8039	,	0.5961	,	0.2471	}
\definecolor{usuko}{rgb}	{	0.8745	,	0.7373	,	0.4824	}
\definecolor{kuchiba}{rgb}	{	0.8157	,	0.6157	,	0.3216	}
\definecolor{kincha}{rgb}	{	0.7725	,	0.4902	,	0.2000	}
\definecolor{chozizome}{rgb}	{	0.5686	,	0.3961	,	0.1608	}
\definecolor{kitsune}{rgb}	{	0.6392	,	0.4431	,	0.1804	}
\definecolor{hushizome}{rgb}	{	0.5804	,	0.4353	,	0.1608	}
\definecolor{kyara}{rgb}	{	0.4510	,	0.3373	,	0.1255	}
\definecolor{susutake}{rgb}	{	0.4588	,	0.3529	,	0.1176	}
\definecolor{shirocha}{rgb}	{	0.7529	,	0.6588	,	0.4118	}
\definecolor{odo}{rgb}	        {	0.7137	,	0.6039	,	0.3137	}
\definecolor{ginsusutake}{rgb}	{	0.5608	,	0.4745	,	0.2392	}
\definecolor{kigaracha}{rgb}	{	0.7490	,	0.6196	,	0.2745	}
\definecolor{kobicha}	{rgb}	{	0.5020	,	0.4118	,	0.1765	}
\definecolor{usuki}	{rgb}	{	0.8549	,	0.7804	,	0.4275	}
\definecolor{yamabuki}	{rgb}	{	0.9098	,	0.8000	,	0.2039	}
\definecolor{tamago}	{rgb}	{	0.8275	,	0.7412	,	0.3373	}
\definecolor{hajizome}	{rgb}	{	0.7412	,	0.6039	,	0.2431	}
\definecolor{yamabukicha}{rgb}	{	0.7059	,	0.5765	,	0.2275	}
\definecolor{kuwazome}	{rgb}	{	0.6471	,	0.5255	,	0.2118	}
\definecolor{namakabe}	{rgb}	{	0.5569	,	0.4549	,	0.1843	}
\definecolor{kuchinashi}{rgb}	{	0.8078	,	0.7333	,	0.3843	}
\definecolor{tomorokoshi}{rgb}	{	0.7804	,	0.7020	,	0.2941	}
\definecolor{shirotsurubami}{rgb}	{	0.8745	,	0.8235	,	0.5922	}
\definecolor{kitsurubami}{rgb}	{	0.6863	,	0.6039	,	0.2118	}
\definecolor{toou}	{rgb}	{	0.8353	,	0.7725	,	0.4235	}
\definecolor{hanaba}	{rgb}	{	0.8549	,	0.8000	,	0.4941	}
\definecolor{torinoko}	{rgb}	{	0.8431	,	0.8118	,	0.6902	}
\definecolor{ukon}	{rgb}	{	0.8039	,	0.7216	,	0.3059	}
\definecolor{kikuchiba}	{rgb}	{	0.7765	,	0.6667	,	0.2902	}
\definecolor{rikyushiracha}{rgb}{	0.6627	,	0.6471	,	0.4784	}
\definecolor{rikyucha}	{rgb}	{	0.5176	,	0.5176	,	0.2588	}
\definecolor{aku}	{rgb}	{	0.5294	,	0.5059	,	0.4039	}
\definecolor{higosusutake}{rgb}	{	0.4863	,	0.4510	,	0.3059	}
\definecolor{rokocha}	{rgb}	{	0.4157	,	0.4353	,	0.3686	}
\definecolor{mirucha}	{rgb}	{	0.4471	,	0.4471	,	0.3725	}
\definecolor{natane}	{rgb}	{	0.7216	,	0.6863	,	0.3882	}
\definecolor{kimirucha}	{rgb}	{	0.5373	,	0.5059	,	0.2549	}
\definecolor{uguisucha}	{rgb}	{	0.4118	,	0.3882	,	0.1961	}
\definecolor{nanohana}	{rgb}	{	0.9882	,	0.9882	,	0.3804	}
\definecolor{kariyasu}	{rgb}	{	0.8039	,	0.7686	,	0.3882	}
\definecolor{kihada}	{rgb}	{	0.9608	,	0.9137	,	0.2863	}
\definecolor{zoge}	{rgb}	{	0.8863	,	0.8235	,	0.7216	}
\definecolor{wara}	{rgb}	{	0.7882	,	0.7490	,	0.4706	}
\definecolor{macha}	{rgb}	{	0.6235	,	0.6392	,	0.4863	}
\definecolor{yamabato}	{rgb}	{	0.5137	,	0.5176	,	0.4000	}
\definecolor{mushikuri}	{rgb}	{	0.8196	,	0.8000	,	0.6118	}
\definecolor{aokuchiba}	{rgb}	{	0.6275	,	0.6392	,	0.3451	}
\definecolor{hiwacha}	{rgb}	{	0.6784	,	0.6863	,	0.4196	}
\definecolor{ominaeshi}	{rgb}	{	0.8745	,	0.8863	,	0.4039	}
\definecolor{wasabi}	{rgb}	{	0.5922	,	0.6824	,	0.5765	}
\definecolor{uguisu}	{rgb}	{	0.3333	,	0.4118	,	0.0627	}
\definecolor{hiwa}	{rgb}	{	0.7098	,	0.7216	,	0.2392	}
\definecolor{aoshirotsurubami}{rgb}	{	0.5961	,	0.6471	,	0.4471	}
\definecolor{yanagicha}	{rgb}	{	0.5373	,	0.5725	,	0.3529	}
\definecolor{rikancha}	{rgb}	{	0.3333	,	0.4196	,	0.2863	}
\definecolor{aikobicha}	{rgb}	{	0.2706	,	0.3412	,	0.2353	}
\definecolor{koke}	{rgb}	{	0.4941	,	0.5686	,	0.3922	}
\definecolor{miru}	{rgb}	{	0.1529	,	0.3216	,	0.1843	}
\definecolor{sensai}	{rgb}	{	0.1373	,	0.2863	,	0.1608	}
\definecolor{baiko}	{rgb}	{	0.5529	,	0.6157	,	0.3922	}
\definecolor{iwai}	{rgb}	{	0.3059	,	0.4118	,	0.2784	}
\definecolor{hiwamoegi}	{rgb}	{	0.4431	,	0.6824	,	0.2431	}
\definecolor{yanagisusutake}{rgb}	{	0.2039	,	0.3333	,	0.1255	}
\definecolor{urayanagi}	{rgb}	{	0.5843	,	0.6824	,	0.2431	}
\definecolor{usumoegi}	{rgb}	{	0.4824	,	0.6706	,	0.2275	}
\definecolor{yanagizome}{rgb}	{	0.4353	,	0.6000	,	0.2078	}
\definecolor{moegi}	{rgb}	{	0.3020	,	0.5961	,	0.1882	}
\definecolor{aoni}	{rgb}	{	0.1216	,	0.4196	,	0.2431	}
\definecolor{matsuba}	{rgb}	{	0.1098	,	0.3686	,	0.2118	}
\definecolor{usuao}	{rgb}	{	0.5529	,	0.7059	,	0.6118	}
\definecolor{wakatake}	{rgb}	{	0.3843	,	0.6824	,	0.5059	}
\definecolor{yanaginezumi}{rgb}	{	0.5020	,	0.6078	,	0.5490	}
\definecolor{oitake}	{rgb}	{	0.4157	,	0.5412	,	0.4392	}
\definecolor{sensaimidori}{rgb}	{	0.2549	,	0.3882	,	0.2627	}
\definecolor{midori}	{rgb}	{	0.0000	,	0.4824	,	0.0000	}
\definecolor{byakuroku}	{rgb}	{	0.6078	,	0.7333	,	0.6196	}
\definecolor{sabiseiji}	{rgb}	{	0.5333	,	0.6588	,	0.5804	}
\definecolor{rokusho}	{rgb}	{	0.3373	,	0.6039	,	0.4039	}
\definecolor{tokusa}	{rgb}	{	0.2549	,	0.4706	,	0.3922	}
\definecolor{onandocha}	{rgb}	{	0.2039	,	0.3725	,	0.3098	}
\definecolor{aotake}	{rgb}	{	0.1412	,	0.5176	,	0.4353	}
\definecolor{rikyunezumi}{rgb}	{	0.4157	,	0.5686	,	0.5490	}
\definecolor{birodo}	{rgb}	{	0.1059	,	0.4275	,	0.3608	}
\definecolor{mishiao}	{rgb}	{	0.1098	,	0.4588	,	0.3882	}
\definecolor{aimirucha}	{rgb}	{	0.2510	,	0.4000	,	0.3608	}
\definecolor{tonotya}	{rgb}	{	0.3059	,	0.4941	,	0.4392	}
\definecolor{mizuasagi}	{rgb}	{	0.1569	,	0.6745	,	0.6745	}
\definecolor{seji}	{rgb}	{	0.4745	,	0.6588	,	0.5922	}
\definecolor{seheki}	{rgb}	{	0.0588	,	0.5569	,	0.4824	}
\definecolor{sabitetsu}	{rgb}	{	0.0392	,	0.3294	,	0.2784	}
\definecolor{tetsu}	{rgb}	{	0.0627	,	0.3961	,	0.3961	}
\definecolor{omeshicha}	{rgb}	{	0.0667	,	0.4667	,	0.4667	}
\definecolor{korainando}{rgb}	{	0.0588	,	0.3922	,	0.0392	}
\definecolor{minatonezumi}{rgb}	{	0.4549	,	0.6000	,	0.6118	}
\definecolor{aonibi}	{rgb}	{	0.1804	,	0.3725	,	0.3882	}
\definecolor{tetsuonando}{rgb}	{	0.1961	,	0.4118	,	0.4275	}
\definecolor{mizu}	{rgb}	{	0.5451	,	0.7412	,	0.7608	}
\definecolor{sabiasagi}	{rgb}	{	0.4510	,	0.6000	,	0.6000	}
\definecolor{kamenozoki}{rgb}	{	0.7098	,	0.9020	,	0.8902	}
\definecolor{asagi}	{rgb}	{	0.3020	,	0.6784	,	0.7098	}
\definecolor{shinbashi}	{rgb}	{	0.0000	,	0.6000	,	0.6353	}
\definecolor{sabionando}{rgb}	{	0.2392	,	0.3961	,	0.3843	}
\definecolor{ainezumi}	{rgb}	{	0.2235	,	0.4000	,	0.3843	}
\definecolor{ai}	{rgb}	{	0.2039	,	0.3765	,	0.4314	}
\definecolor{onando}	{rgb}	{	0.1843	,	0.3686	,	0.4000	}
\definecolor{hanaasagi}	{rgb}	{	0.2000	,	0.6196	,	0.7098	}
\definecolor{chigusa}	{rgb}	{	0.1922	,	0.5725	,	0.6745	}
\definecolor{masuhana}	{rgb}	{	0.1569	,	0.4667	,	0.5569	}
\definecolor{hanada}	{rgb}	{	0.0039	,	0.4275	,	0.5373	}
\definecolor{noshimehana}{rgb}	{	0.1412	,	0.4235	,	0.5098	}
\definecolor{omeshionando}{rgb}	{	0.1176	,	0.3529	,	0.4235	}
\definecolor{sora}	{rgb}	{	0.1451	,	0.7216	,	0.8039	}
\definecolor{konpeki}	{rgb}	{	0.0902	,	0.5098	,	0.7333	}
\definecolor{kurotsurubami}{rgb}{	0.0627	,	0.3137	,	0.3451	}
\definecolor{gunjo}	{rgb}	{	0.5098	,	0.7882	,	0.9098	}
\definecolor{kon}	{rgb}	{	0.0000	,	0.2000	,	0.4000	}
\definecolor{kachi}	{rgb}	{	0.0000	,	0.1765	,	0.3490	}
\definecolor{ruri}	{rgb}	{	0.0078	,	0.3922	,	0.6510	}
\definecolor{konjo}	{rgb}	{	0.0039	,	0.3137	,	0.5216	}
\definecolor{rurikon}	{rgb}	{	0.0000	,	0.3020	,	0.5020	}
\definecolor{benimidori}{rgb}	{	0.3373	,	0.4588	,	0.6784	}
\definecolor{konkikyo}	{rgb}	{	0.0039	,	0.0667	,	0.4275	}
\definecolor{hujinezumi}{rgb}	{	0.3216	,	0.3686	,	0.6118	}
\definecolor{benikakehana}{rgb}	{	0.2275	,	0.2471	,	0.5882	}
\definecolor{hujiiro}	{rgb}	{	0.4706	,	0.4863	,	0.7059	}
\definecolor{hutaai}	{rgb}	{	0.3137	,	0.3333	,	0.5608	}
\definecolor{hujimurasaki}{rgb}	{	0.4157	,	0.3333	,	0.5686	}
\definecolor{kikyo}	{rgb}	{	0.3529	,	0.3216	,	0.5725	}
\definecolor{shion}	{rgb}	{	0.5137	,	0.4196	,	0.6784	}
\definecolor{messhi}	{rgb}	{	0.2196	,	0.1765	,	0.3098	}
\definecolor{shikon}	{rgb}	{	0.2510	,	0.1961	,	0.3529	}
\definecolor{kokimurasaki}{rgb}	{	0.3098	,	0.2000	,	0.3765	}
\definecolor{usu}	{rgb}	{	0.7294	,	0.6353	,	0.7843	}
\definecolor{hashita}	{rgb}	{	0.6000	,	0.4549	,	0.6706	}
\definecolor{rindo}	{rgb}	{	0.5922	,	0.5137	,	0.7529	}
\definecolor{sumire}	{rgb}	{	0.3882	,	0.2157	,	0.5922	}
\definecolor{nasukon}	{rgb}	{	0.3216	,	0.2235	,	0.4353	}
\definecolor{murasaki}	{rgb}	{	0.3137	,	0.1765	,	0.4824	}
\definecolor{kurobeni}	{rgb}	{	0.2353	,	0.1294	,	0.3059	}
\definecolor{ayame}	{rgb}	{	0.4275	,	0.1569	,	0.5098	}
\definecolor{benihuji}	{rgb}	{	0.6824	,	0.5333	,	0.6667	}
\definecolor{edomurasaki}{rgb}	{	0.3961	,	0.1529	,	0.4392	}
\definecolor{kodaimurasaki}{rgb}{	0.4353	,	0.2863	,	0.4627	}
\definecolor{shikon}{rgb}       {	0.4118	,	0.2549	,	0.3765	}
\definecolor{hatobanezumi}{rgb}	{	0.4235	,	0.3804	,	0.4392	}
\definecolor{budonezumi}{rgb}	{	0.3412	,	0.2196	,	0.3529	}
\definecolor{ebizome}	{rgb}	{	0.4000	,	0.2235	,	0.3882	}
\definecolor{hujisusutake}{rgb}	{	0.2549	,	0.0824	,	0.2471	}
\definecolor{usuebi}	{rgb}	{	0.6549	,	0.4392	,	0.6039	}
\definecolor{botan}	{rgb}	{	0.6706	,	0.0392	,	0.4353	}
\definecolor{umemurasaki}{rgb}	{	0.6667	,	0.3922	,	0.5176	}
\definecolor{nisemurasaki}{rgb}	{	0.3686	,	0.0000	,	0.3216	}
\definecolor{murasakitobi}{rgb}	{	0.3922	,	0.3255	,	0.3529	}
\definecolor{ususuo}{rgb}	{	0.7569	,	0.4000	,	0.5412	}
\definecolor{suo}{rgb}	        {	0.6941	,	0.0627	,	0.4157	}
\definecolor{kuwanomi}{rgb}	{	0.4196	,	0.0549	,	0.2745	}
\definecolor{nibi}{rgb}	        {	0.4549	,	0.4235	,	0.3961	}
\definecolor{benikeshi}{rgb}	{	0.4118	,	0.3804	,	0.3843	}
\definecolor{shironeri}{rgb}	{	0.9843	,	0.9843	,	0.9843	}
\definecolor{shironezumi}{rgb}	{	0.6471	,	0.6588	,	0.6627	}
\definecolor{ginnezumi}{rgb}	{	0.5255	,	0.5490	,	0.5922	}
\definecolor{sunezumi}{rgb}	{	0.4275	,	0.4392	,	0.4392	}
\definecolor{dobunezumi}{rgb}	{	0.2863	,	0.2941	,	0.2941	}
\definecolor{aisumicha}{rgb}	{	0.2078	,	0.2196	,	0.2510	}
\definecolor{binrojizome}{rgb}	{	0.2118	,	0.0824	,	0.0706	}
\definecolor{sumizome}{rgb}	{	0.2706	,	0.2706	,	0.2706	}

\newcommand{\R}{\mathbb{R}}
\newcommand{\Z}{\mathbb{Z}}

\newcommand{\Span}{\mathrm{span}}
\newcommand{\odeg}{\overline{\deg}\,}
\newcommand{\udeg}{\underline{\deg}\,}

\newcommand{\br}[1]{\left\langle #1 \right\rangle} 

\theoremstyle{definition}
\newtheorem{defi}{Definition}[section]
\newtheorem{theo}[defi]{Theorem}
\newtheorem{lem}[defi]{Lemma}
\newtheorem{rem}[defi]{Remark}

\newtheorem{cor}[defi]{Corollary}
\newtheorem{exam}[defi]{Example}

\numberwithin{equation}{section}
\numberwithin{figure}{section}

\title[]{Span of the Jones polynomials of certain v-adequate virtual links}
\author{Minori Okamura \and Keiichi Sakai}
\address{Faculty of Mathematics, Shinshu University, 3-1-1 Asahi, Matsumoto, Nagano 390-8621, Japan}
\email{17ss104c@gmail.com}
\email{ksakai@math.shinshu-u.ac.jp}
\date{\today}
\begin{document}
\begin{abstract}
It is known that the Kauffman-Murasugi-Thistlethwaite type inequality becomes an equality for any (possibly virtual) adequate link diagram.
We refine this condition.
As an application we obtain a criterion for virtual link diagram with exactly one virtual crossing to represent a properly virtual link.
\end{abstract}
\maketitle
\section{Introduction}
The Kauffman-Murasugi-Thistlethwaite inequality \cite{Kauffman87,Murasugi87,Thistlethwaite87}, KMT inequality for short, is known as an effective tool to estimate, and in some cases determine, the minimal crossing number of (classical) links in terms of the span of the Jones polynomial (or equivalently of the Kauffman bracket polynomial).
The KMT inequality is a strict inequality for some links, and the defect is closely related to the Euler characteristics of the \emph{Turaev surface} \cite{Turaev87} (also known as the \emph{atom} \cite{ManturovIlyutko10}).
Thus the KMT inequality has a refined form involving the Euler characteristics of the Turaev surfaces (Theorem \ref{Atom_ThAtomProperSpan}; see also \cite{BaeMorton03,DFKLS08}).
Moreover it is known that the refined KMT inequality becomes an equality for \emph{adequate} link diagrams.
The proofs of these results seem easier than that of the original KMT inequality, and in fact the refined inequality provides a simple proof of the Tait conjecture \cite{Turaev87}.

In this paper we refine this sufficient condition for (possibly virtual) link diagram under which condition the refined KMT inequality becomes an equality.
We moreover introduce the notion of \emph{v-adequate link diagrams} to be diagrams obtained by virtualizing exactly one real crossing of some adequate diagrams, and as an application we prove that the refined KMT inequality becomes an equality for v-adequate link diagrams with certain condition.
This means that we can determine the span of the Jones polynomials of an adequate diagram and a v-adequate one obtained from the former,
and this allows us to show that one of the span of the Jones polynomials of these links cannot be divided by four.
We therefore obtain a recipe for producing properly virtual links.
The ``certain condition'' is valid for v-adequate diagrams derived from classical adequate diagrams,
and hence we indeed have infinitely many examples to which our criterion is applied.
These examples can be seen as generalizations of Kishino's result \cite{Kishino00}.

This paper is organized as follows.
In \S\ref{IntroVK} we review the Kauffman bracket and the Turaev surface of (possibly virtual) link diagrams.
We introduce the notion of (pseudo-) adequacy in \S\ref{s:Adequacy}, and we prove our main theorems in \S\ref{s:proofs}.

\section{Preliminaries} \label{IntroVK}
We follow the conventions in \cite{Kauffman99} for virtual link diagrams.
%
Two virtual link diagrams $D$ and $D'$ are said to be equivalent
if $D$ can be transformed into $D'$ by a finite sequence of 
Reidemeister moves and virtual Reidemeister moves \cite[Figure 2]{Kauffman99}.

%
%
%

\subsection{The Kauffman bracket polynomial}
The transformations of a virtual link diagram in a neighborhood of a real crossing 
as in Figure \ref{Atom_2_image_AB-splice} are called respectively \emph{A-splice} and \emph{B-splice}.
We draw a dotted line at each spliced real crossing as in Figure \ref{Atom_2_image_AB-splice}, 
which we call a \emph{connecting arc}. 
	\begin{figure}
	\centering
{\unitlength 0.1in%
\begin{picture}(41.0000,8.0000)(11.0000,-32.0000)%
%
\special{pn 8}%
\special{pa 3800 2800}%
\special{pa 4100 2800}%
\special{fp}%
\special{sh 1}%
\special{pa 4100 2800}%
\special{pa 4033 2780}%
\special{pa 4047 2800}%
\special{pa 4033 2820}%
\special{pa 4100 2800}%
\special{fp}%
\put(39.5000,-26.5000){\makebox(0,0){B-splice}}%
%
\special{pn 8}%
\special{pa 2400 2800}%
\special{pa 2100 2800}%
\special{fp}%
\special{sh 1}%
\special{pa 2100 2800}%
\special{pa 2167 2820}%
\special{pa 2153 2800}%
\special{pa 2167 2780}%
\special{pa 2100 2800}%
\special{fp}%
\put(22.5000,-26.5000){\makebox(0,0){A-splice}}%
%
\special{pn 13}%
\special{pa 4800 2500}%
\special{pa 4800 3100}%
\special{dt 0.045}%
%
\special{pn 13}%
\special{pa 1200 2800}%
\special{pa 1800 2800}%
\special{dt 0.045}%
%
\special{pn 13}%
\special{pa 2700 2400}%
\special{pa 3000 2700}%
\special{fp}%
%
\special{pn 13}%
\special{pa 3200 2900}%
\special{pa 3500 3200}%
\special{fp}%
%
\special{pn 13}%
\special{pa 3500 2400}%
\special{pa 2700 3200}%
\special{fp}%
%
\special{pn 13}%
\special{ar 2300 2800 566 566 2.3561945 3.9269908}%
%
\special{pn 13}%
\special{ar 700 2800 566 566 5.4977871 0.7853982}%
%
\special{pn 13}%
\special{ar 4800 2100 500 500 0.6435011 2.4980915}%
%
\special{pn 13}%
\special{ar 4800 3500 500 500 3.7850938 5.6396842}%
\end{picture}}%
	\caption{A-splice and B-splice, and connecting arcs}
	\label{Atom_2_image_AB-splice}
	\end{figure}
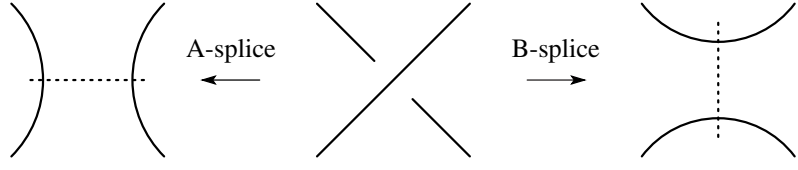

Let $D$ be a virtual link diagram. 
A \emph{state} $s$ of $D$   is a map from the set of the real crossings of $D$ to the set $\{\mathrm{A},\mathrm{B}\}$. 
We denote by $\mathcal{S}$ the set of the states of $D$. 
For $s\in\mathcal{S}$, 
let $D(s)$ be the virtual link diagram obtained from $D$ by performing $s(p)$-splice at each real crossing $p$ of $D$.
Define the \emph{weight} of $s\in\mathcal{S}$, denoted by $\br{D/s}$, by 
	\[
	\br{D/s} \coloneqq  A^{\alpha(s)-\beta(s)} (-A^2-A^{-2})^{\sharp D(s) -1 }\in\Z[A,A^{-1}],
	\]
where $\alpha(s)\coloneqq \sharp s^{-1}(A)$, $\beta(s)\coloneqq \sharp s^{-1}(B)$ and $\sharp D(s)$ is the number of components of $D(s)$.
The \emph{Kauffman bracket polynomial} $\br{D}$ of $D$ is defined by
	\[
	\br{D} \coloneqq  \sum_{s \in \mathcal{S}} \br{D/s}.
	\]
%
%
%
For a Laurent polynomial $f\in\Z[A,A^{-1}]$ we denote the maximal (resp.\ the minimal) degree of $f$ by $\odeg f$ (resp.\ $\udeg f$),
and define $\Span (f)$ as
\[ 
\Span (f)\coloneqq \odeg f-\udeg f.
\]
It is well known that $\Span\br{D}$ is invariant under the generalized Reidemeister moves.
For a virtual link $L$ we define $\Span(L)$ as $\Span\br{D}$ for any diagram $D$ representing $L$. 
The following property is well known.
%
%
\begin{lem}
\label{ClassicalLemMultipleOf4}
If $L$ is a classical link, 
then $\Span(L)$ can be divided by $4$. 
\end{lem}
\subsection{The Turaev surface}

Recall the definition of the Turaev surface of a virtual link diagram $D$ with $c(D) > 0$. 
We replace all the real crossings of $D$ with disks and join these disks with bands, 
each of which corresponds to an arc of $D$, so that the resulting surface $F'(D)$ is orientable and is almost embedded into $\R^2$, 
except for the neighborhoods of virtual crossings 
(Figure \ref{Atom_1_image_Atom}). 
The surface $F'(D)$ with boundary has been introduced in \cite{Kamada04} and is called the \emph{fat frame} of $D$.
    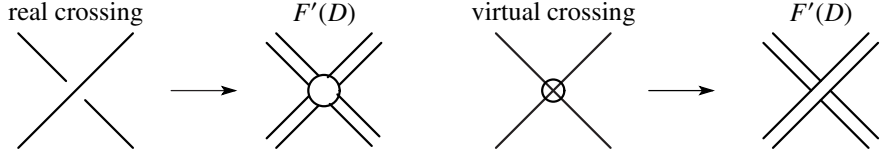
\begin{figure}
	\centering
{\unitlength 0.1in%
\begin{picture}(45.3500,7.7300)(5.6500,-18.0000)%
%
\special{pn 13}%
\special{pa 3100 1800}%
\special{pa 3700 1200}%
\special{fp}%
%
\special{pn 13}%
\special{pa 3700 1800}%
\special{pa 3100 1200}%
\special{fp}%
%
\special{pn 8}%
\special{pa 1400 1500}%
\special{pa 1735 1500}%
\special{fp}%
\special{sh 1}%
\special{pa 1735 1500}%
\special{pa 1668 1480}%
\special{pa 1682 1500}%
\special{pa 1668 1520}%
\special{pa 1735 1500}%
\special{fp}%
\put(9.0000,-11.0000){\makebox(0,0){real crossing}}%
%
\special{pn 13}%
\special{pa 1950 1200}%
\special{pa 2170 1420}%
\special{fp}%
%
\special{pn 13}%
\special{ar 2200 1500 82 82 0.0000000 6.2831853}%
%
\special{pn 13}%
\special{pa 1900 1750}%
\special{pa 2120 1530}%
\special{fp}%
%
\special{pn 13}%
\special{pa 4500 1750}%
\special{pa 5050 1200}%
\special{fp}%
%
\special{pn 13}%
\special{pa 5050 1800}%
\special{pa 4800 1550}%
\special{fp}%
\put(22.0000,-11.0000){\makebox(0,0){$F'(D)$}}%
%
\special{pn 13}%
\special{ar 3400 1500 58 58 0.0000000 6.2831853}%
%
\special{pn 8}%
\special{pa 3900 1500}%
\special{pa 4235 1500}%
\special{fp}%
\special{sh 1}%
\special{pa 4235 1500}%
\special{pa 4168 1480}%
\special{pa 4182 1500}%
\special{pa 4168 1520}%
\special{pa 4235 1500}%
\special{fp}%
\put(34.0000,-11.0000){\makebox(0,0){virtual crossing}}%
\put(48.0000,-11.0000){\makebox(0,0){$F'(D)$}}%
%
\special{pn 13}%
\special{pa 1200 1200}%
\special{pa 600 1800}%
\special{fp}%
%
\special{pn 13}%
\special{pa 600 1200}%
\special{pa 850 1450}%
\special{fp}%
%
\special{pn 13}%
\special{pa 950 1550}%
\special{pa 1200 1800}%
\special{fp}%
%
\special{pn 13}%
\special{pa 1950 1800}%
\special{pa 2170 1580}%
\special{fp}%
%
\special{pn 13}%
\special{pa 2280 1470}%
\special{pa 2500 1250}%
\special{fp}%
%
\special{pn 13}%
\special{pa 2220 1430}%
\special{pa 2450 1200}%
\special{fp}%
%
\special{pn 13}%
\special{pa 1900 1250}%
\special{pa 2120 1470}%
\special{fp}%
%
\special{pn 13}%
\special{pa 2270 1520}%
\special{pa 2490 1740}%
\special{fp}%
%
\special{pn 13}%
\special{pa 2230 1570}%
\special{pa 2450 1790}%
\special{fp}%
%
\special{pn 13}%
\special{pa 4550 1800}%
\special{pa 5100 1250}%
\special{fp}%
%
\special{pn 13}%
\special{pa 5100 1750}%
\special{pa 4850 1500}%
\special{fp}%
%
\special{pn 13}%
\special{pa 4750 1500}%
\special{pa 4500 1250}%
\special{fp}%
%
\special{pn 13}%
\special{pa 4800 1450}%
\special{pa 4550 1200}%
\special{fp}%
\end{picture}}%
	\caption{Fat frame $F'(D)$}
	\label{Atom_1_image_Atom}
   \end{figure}
Around the real crossings of $D$, 
we color $\partial F'(D)$ in checkerboard manner 
so that the components corresponding to curves obtained by A-splice (resp.\ B-splice) is colored by red (resp.\ blue),
as in Figure \ref{image_coloredBands}. 
	\begin{figure}
	\centering
	\includegraphics[scale=0.4]{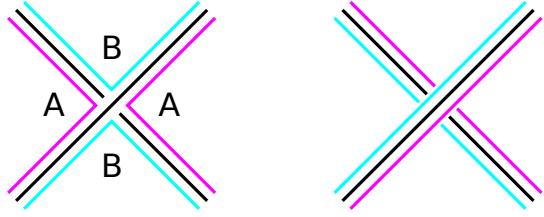}
	\caption{Checkerboard coloring of $F'(D)$ around the real crossings of $D$}
	\label{image_coloredBands}
	\end{figure}
We twist each band of $F'(D)$ if necessary
so that the checkerboard colorings of $\partial F'(D)$ near the real crossings extend consistently to whole boundary,
as in Figure \ref{Atom_1_image_ExamAtom}.
	\begin{figure}
	\includegraphics[scale=0.4]{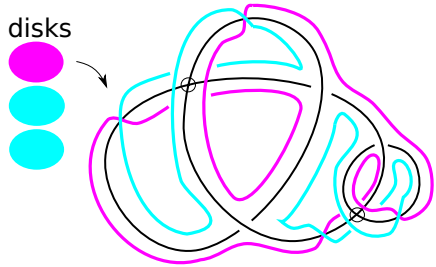}
	\caption{A link diagram with unorientable Turaev surface}
	\label{Atom_1_image_ExamAtom}
	\end{figure}
The resulting surface with boundary is called the \emph{twisted fat frame} of $D$, and denoted by $F(D)$. 
Now let $T_D$ be the closed surface obtained by attaching disks along $\partial F(D)$, and call $T_D$ the Turaev surface of $D$.
%
%
\begin{rem}
By definition the Turaev surface of a trivial link diagram is a union of $S^2$. 
\end{rem}
\begin{rem}
$T_D$ is orientable for any classical link diagram \cite{Turaev87}.
This is not necessarily the case for virtual diagrams; 
see Figure \ref{Atom_1_image_ExamAtom} for example. 
It is not hard to see that $F(D) = F'(D)$, and hence $T_D$ is orientable, if $D$ is alternating.
\end{rem} 
\begin{defi}
For a virtual link diagram $D$, 
let $s_A$ (resp.\ $s_B$) be the state of $D$ 
that maps every real crossing of $D$ to A (resp.\ B). 
\end{defi}
\begin{rem}
\label{Atom2_LemLoopsAsB}
By construction,
the red boundary of $\partial F(D)$ coresspond to $D(s_A)$, 
and the blue boundary of $\partial F(D)$ coresspond to $D(s_B)$. 
Thus we have a one-to-one corespondence between the components of $D(s_A) \cup D(s_B)$ and those of $\partial F(D)$. 
\end{rem}

%
%
\begin{lem}
\label{Atom_LemEulercha}
For a virtual link diagram $D$, 
let $\chi(D)$ be the Euler characteristic of $T_D$. 
Then we have $\sharp D(s_A) + \sharp D(s_B) = \chi(D) + c(D)$.
\end{lem}
\begin{proof} 
$T_D$ can be decomposed into a cell complex; 
0-cells are in one-one correspondence to the real crossings of $D$, 
1-cells correspond to the arcs of $D$, 
and 2-cells corresspond to the disks attached to $F(D)$ along the boundary. 
The number of arcs is equal to $2c(D)$, 
since $D$ can be seen as a 4-valent graph whose vertices are the real crossings of $D$ and whose edges are the arcs of $D$. 
Since the number of 2-cells is $\sharp D(s_A) + \sharp D(s_B)$ as mentioned in Remark \ref{Atom2_LemLoopsAsB}, 
we have 
	\begin{align*}
	\chi(D) 
	&= c(D) - 2c(D) + \sharp D(s_A)  + \sharp D(s_B)  \\
	&= -c(D) + \sharp D(s_A) + \sharp D(s_B) . \qedhere
	\end{align*}
\end{proof}
%
%

\section{Pseudo-adequate daigrams and v-pseudo-adequate diagrams}\label{s:Adequacy}
%
\begin{defi}
\label{DefSemiAdequate}
A virtual link diagram $D$ is said to be \emph{pseudo-adequate} if, for any $s\in\mathcal{S}$,
\begin{enumerate}[(a)]
\item
	$\sharp D(s_A) \ge \sharp D(s_A(1))$ and
\item
	$\sharp D(s_B) \ge \sharp D(s_B(1))$,
\end{enumerate}
where $s_A(1)$ (resp.\ $s_B(1)$) is a state that sends one real crossing to B (resp.\ A) and all the other crossings to A (resp.\ B); 
for notations see \S\ref{s:proofs}.
\end{defi}
\begin{rem}
A virtual link diagram $D$ is said to be \emph{A-adequate} (resp.\ \emph{B-adequate}) 
if the condition (a) (resp.\ (b)) in Definition \ref{DefSemiAdequate} holds after replacing ``$\ge$'' with ``$>$'',
and \emph{adequate} if $D$ is A-adequate and B-adequate \cite{LickorishThistlethwaite88}.
$D$ is adequate if and only if the four components of $\partial F(D)$ near each real crossing $p$ are all distinct.

Clearly an adequate diagram is pseudo-adequate.
Any pseudo-adequate classical diagram is adequate, 
but this is not the case for virtual diagrams; 
the diagram $H_n$ in Figure \ref{image_Pseudo-adequate} is pseudo-adequate but not adequate,
because $\sharp H_n(s_A)=\sharp H_n(s_A(1))=1$ and $\sharp H_n(s_B)=\sharp H_n(s_B(1))=1$.
\end{rem}
	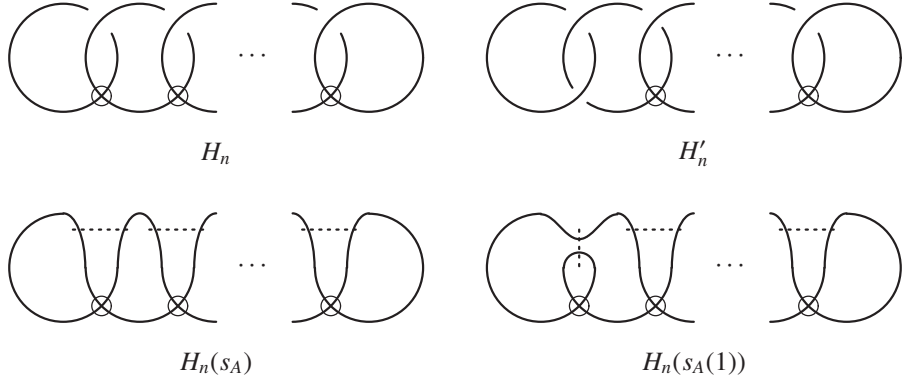
\begin{figure}
	\centering
{\unitlength 0.1in%
\begin{picture}(46.6600,18.1000)(5.1700,-25.2700)%
%
\special{pn 13}%
\special{ar 1600 1000 283 283 1.5707963 4.7123890}%
%
\special{pn 13}%
\special{ar 800 1000 283 283 5.8195377 5.1760366}%
%
\special{pn 8}%
\special{ar 1000 1200 51 51 0.0000000 6.2831853}%
%
\special{pn 13}%
\special{ar 1200 1000 283 283 5.8195377 5.1760366}%
%
\special{pn 13}%
\special{ar 2400 1000 283 283 0.0000000 6.2831853}%
%
\special{pn 8}%
\special{ar 1400 1200 51 51 0.0000000 6.2831853}%
%
\special{pn 13}%
\special{ar 2000 1000 283 283 4.7123890 5.1760366}%
%
\special{pn 13}%
\special{ar 2000 1000 283 283 5.8195377 1.5707963}%
\put(18.0000,-10.0000){\makebox(0,0){$\cdots$}}%
%
\special{pn 8}%
\special{ar 2200 1200 51 51 0.0000000 6.2831853}%
\put(16.0000,-15.0000){\makebox(0,0){$H_n$}}%
%
\special{pn 13}%
\special{ar 800 2100 283 283 6.2831853 4.7123890}%
%
\special{pn 13}%
\special{ar 1600 2100 283 283 1.5707963 3.1415927}%
%
\special{pn 13}%
\special{ar 1200 2100 115 285 3.1415927 6.2831853}%
\put(18.0000,-21.0000){\makebox(0,0){$\cdots$}}%
%
\special{pn 13}%
\special{ar 2400 2100 283 283 4.7123890 3.1415927}%
%
\special{pn 13}%
\special{ar 2000 2100 283 283 6.2831853 1.5707963}%
\put(16.0000,-26.0000){\makebox(0,0){$H_n(s_A)$}}%
%
\special{pn 8}%
\special{ar 1000 2300 51 51 0.0000000 6.2831853}%
%
\special{pn 8}%
\special{ar 1400 2300 51 51 0.0000000 6.2831853}%
%
\special{pn 8}%
\special{ar 2200 2300 51 51 0.0000000 6.2831853}%
%
\special{pn 13}%
\special{ar 3300 2100 283 283 6.2831853 4.7123890}%
%
\special{pn 13}%
\special{ar 3700 2100 283 283 6.2831853 3.1415927}%
%
\special{pn 13}%
\special{ar 4100 2100 283 283 1.5707963 3.1415927}%
\put(43.0000,-21.0000){\makebox(0,0){$\cdots$}}%
%
\special{pn 13}%
\special{ar 4900 2100 283 283 4.7123890 3.1415927}%
%
\special{pn 13}%
\special{ar 3500 2100 80 80 3.1415927 6.2831853}%
%
\special{pn 8}%
\special{ar 3500 2300 51 51 0.0000000 6.2831853}%
%
\special{pn 8}%
\special{ar 3900 2300 51 51 0.0000000 6.2831853}%
%
\special{pn 8}%
\special{ar 4700 2300 51 51 0.0000000 6.2831853}%
\put(41.0000,-26.0000){\makebox(0,0){$H_n(s_A(1))$}}%
\special{pn 13}%
\special{pa 3300 1820}%
\special{pa 3305 1820}%
\special{pa 3320 1823}%
\special{pa 3330 1827}%
\special{pa 3335 1830}%
\special{pa 3340 1832}%
\special{pa 3345 1836}%
\special{pa 3350 1839}%
\special{pa 3370 1855}%
\special{pa 3430 1915}%
\special{pa 3450 1931}%
\special{pa 3455 1934}%
\special{pa 3460 1938}%
\special{pa 3465 1940}%
\special{pa 3470 1943}%
\special{pa 3480 1947}%
\special{pa 3495 1950}%
\special{pa 3505 1950}%
\special{pa 3520 1947}%
\special{pa 3530 1943}%
\special{pa 3535 1940}%
\special{pa 3540 1938}%
\special{pa 3545 1934}%
\special{pa 3550 1931}%
\special{pa 3570 1915}%
\special{pa 3630 1855}%
\special{pa 3650 1839}%
\special{pa 3655 1836}%
\special{pa 3660 1832}%
\special{pa 3665 1830}%
\special{pa 3670 1827}%
\special{pa 3680 1823}%
\special{pa 3695 1820}%
\special{pa 3700 1820}%
\special{fp}%
%
\special{pn 13}%
\special{ar 3300 1000 283 283 5.8195377 5.1760366}%
%
\special{pn 13}%
\special{ar 3700 1000 283 283 2.5535901 5.1760366}%
%
\special{pn 13}%
\special{ar 4100 1000 283 283 1.5707963 4.7123890}%
\put(43.0000,-10.0000){\makebox(0,0){$\cdots$}}%
%
\special{pn 13}%
\special{ar 4500 1000 283 283 4.7123890 5.1760366}%
%
\special{pn 13}%
\special{ar 4500 1000 283 283 5.8195377 1.5707963}%
%
\special{pn 13}%
\special{ar 4900 1000 283 283 0.0000000 6.2831853}%
%
\special{pn 8}%
\special{ar 4700 1200 51 51 0.0000000 6.2831853}%
%
\special{pn 13}%
\special{ar 3700 1000 283 283 5.6951827 2.1587989}%
\put(41.0000,-15.0000){\makebox(0,0){$H'_n$}}%
%
\special{pn 13}%
\special{ar 800 2100 115 285 4.7123890 6.2831853}%
%
\special{pn 13}%
\special{ar 1200 2100 283 283 6.2831853 3.1415927}%
%
\special{pn 13}%
\special{ar 2000 2100 115 285 4.7123890 6.2831853}%
%
\special{pn 13}%
\special{ar 2400 2100 115 285 3.1415927 4.7123890}%
%
\special{pn 13}%
\special{ar 1600 2100 115 285 3.1415927 4.7123890}%
%
\special{pn 13}%
\special{ar 3700 2100 115 285 4.7123890 6.2831853}%
%
\special{pn 13}%
\special{ar 4900 2100 115 285 3.1415927 4.7123890}%
%
\special{pn 13}%
\special{ar 4500 2100 283 283 6.2831853 1.5707963}%
%
\special{pn 13}%
\special{ar 4500 2100 115 285 4.7123890 6.2831853}%
%
\special{pn 13}%
\special{pa 850 1900}%
\special{pa 1150 1900}%
\special{dt 0.045}%
%
\special{pn 13}%
\special{pa 1250 1900}%
\special{pa 1550 1900}%
\special{dt 0.045}%
%
\special{pn 13}%
\special{pa 2050 1900}%
\special{pa 2350 1900}%
\special{dt 0.045}%
%
\special{pn 13}%
\special{pa 3750 1900}%
\special{pa 4050 1900}%
\special{dt 0.045}%
%
\special{pn 13}%
\special{ar 4100 2100 115 285 3.1415927 4.7123890}%
%
\special{pn 13}%
\special{pa 4550 1900}%
\special{pa 4850 1900}%
\special{dt 0.045}%
%
\special{pn 13}%
\special{pa 3500 1900}%
\special{pa 3500 2100}%
\special{dt 0.045}%
%
\special{pn 8}%
\special{ar 3900 1200 51 51 0.0000000 6.2831853}%
\end{picture}}%
	\caption{$H_n$ is pseudo-adequate but not adequate, and $H_n$ is v-pseudo-adequate obtained from $H'_n$}
	\label{image_Pseudo-adequate}
	\end{figure}
%
%
%
%

A \emph{virtualization of a real crossing} is the replacement of the real crossing with a virtual crossing.
\begin{defi}\label{def:v-adequate}
A \emph{v-adequate} (resp.\ \emph{v-pseudo-adequate}) diagram is a diagram $D$
obtained from an adequate (resp.\ a pseudo-adequate) diagram $D'$ by virtualizing one real crossing of $D'$.
\end{defi}

\begin{exam}\label{ex:v-pseudo-adequate}
The diagram $H_n$ in Figure \ref{image_Pseudo-adequate} is v-pseudo-adequate;
indeed $H_n$ can be obtained from $H'_n$ by virtualizing one of its real crossings, 
and $H'_n$ is pseudo-adequate because $\sharp H'_n(s_A)=\sharp H'_n(s_B)=2$ while $\sharp H'_n(s_A(1))$ and $\sharp H'_n(s_B(1))$ are $1$ or $2$. 
\end{exam}

%
%
\begin{defi}\label{DefAdmissible}
For $s\in\mathcal{S}$, 
a connecting arc $\gamma$ in $D(s)$ is said to be \emph{admissible} if
	\begin{enumerate}[(1)]
	\item 
	$\gamma$ connect two distinct components of $D(s)$, or
	
	\item
	both the endpoint of $\gamma$ are on a single component of $D(s)$ 
	and any orientation of the component looks as in Figure \ref{image_admissible}.
	\end{enumerate}
	\begin{figure}
	\input{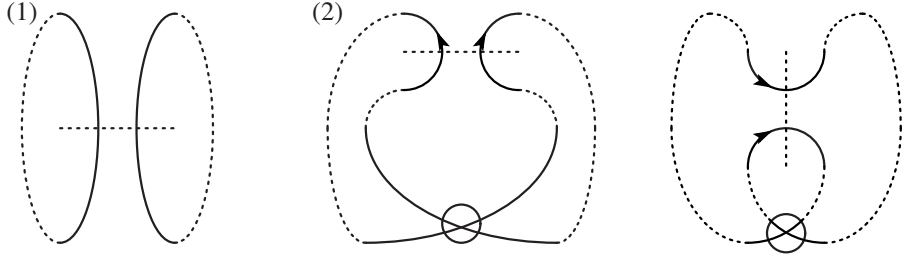}
	\caption{Admissible connecting arcs}
	\label{image_admissible} 
	\end{figure}
\end{defi}

%
\section{The refined KMT inequality}\label{s:proofs}
%
\begin{theo}[\cite{BaeMorton03}, see also \cite{DFKLS08, ManturovIlyutko10}]
\label{Atom_ThAtomProperSpan}
Let $D$ be a virtual link diagram representing a virtual link $L$, 
and $\chi(D)$ be the Euler characteristic of $T_D$. 
Then we have
	\[
	\Span(L) \le 4c(D) + 2(\chi(D) - 2).
	\]
\end{theo}
It is known that the equality holds in Theorem \ref{Atom_ThAtomProperSpan} for adequate diagrams.
This condition can be refined as follows:
\begin{theo}
\label{Atom_SpanPseudoAdequate}
The equality holds in Theorem \ref{Atom_ThAtomProperSpan} for pseudo-adequate diagrams.
\end{theo}
The above Theorem \ref{Atom_SpanPseudoAdequate} deduces our main theorem.
\begin{theo}
\label{Atom_SpanVadequate}
Let $D$ be a v-pseudo-adequate virtual link diagram obtained from 
a pseudo-adequate diagram $D'$ by replacing a real crossing $p$ with a virtual crossing $vp$. 
Let $C_A$ and $C_B$ be the components of $D(s_A)$ and $D(s_B)$ including $vp$ respectively. 
If any connecting arcs of $D(s_A)$ (resp.\ $D(s_B)$) 
whose endpoints are both on $C_A$ (resp.\ $C_B$) are admissible, 
then $\Span\br{D} = 4c(D) + 2(\chi(D) -2)$.
\end{theo}
\begin{cor} \label{CorClassicalVAdequate}
Let $D$ be a v-adequate diagram obtained from an adequate classical diagram $D'$.
Then $D$ does not represent any classical link.
\end{cor}
\begin{proof}
Any connecting arcs of $D(s_A)$ and $D(s_B)$ whose endpoints are both on $C_A$ and $C_B$ 
(Theorem \ref{Atom_SpanVadequate}) are admissible, 
because both $C_A$ and $C_B$ contain exactly one virtual crossing; see Figure \ref{image_admissible}. 
Theorem \ref{Atom_SpanVadequate} therefore implies $\Span\br{D} = 4c(D) + 2(\chi(D) -2)$. 
On the other hand, 
because $D'$ is adequate (and hence pseudo-adequate), 
Theorem \ref{Atom_SpanPseudoAdequate} implies $\Span\br{D'} = 4c(D') + 2(\chi(D')-2)$.
Since moreover $c(D') = c(D) + 1$, $\sharp D'(s_A) = \sharp D(s_A) +1$ and $\sharp D'(s_B) = \sharp D(s_B) +1$,
we obtain $\chi(D') = \chi(D) +1$ by Lemma \ref{Atom_LemEulercha}.
Thus 
	\begin{align*}
	\Span\br{D'} 
		& = 4(c(D) -1) + 2(\chi(D)-1-2) = 4c(D) + 2(\chi(D)-2) -6 \\
		& = \Span\br{D} -6.
	\end{align*}
Since $\Span\br{D'}$ is divisible by $4$ by Lemma \ref{ClassicalLemMultipleOf4}, 
$\Span\br{D}$ cannot be divided by $4$. 
Thus $D$ cannot represent any classical link.
\end{proof}

\begin{rem}
In his master's thesis, Kishino \cite{Kishino00} proved that the $\Span\br{D}=4c(D)-2$ if a diagram $D$ is obtained from a proper alternating classical diagram, and hence $D$ does not represent any classical link.
Corollary \ref{CorClassicalVAdequate} generalizes his result; any proper alternating diagram is adequate.
Theorem \ref{Atom_SpanVadequate} has another application to proper alternating virtual diagrams; see \cite{Kamada04}.
\end{rem}

\begin{exam}
The diagram $H_n$ in Figure \ref{image_Pseudo-adequate} is pseudo-adequate and v-pseudo-adequate.
Moreover the endpoints of any connecting arcs of $H_n(s_A)$ (resp.\ $K_n(s_B)$) are admissible.
Thus both Theorems~\ref{Atom_SpanPseudoAdequate} and \ref{Atom_SpanVadequate} can be applied to deduce $\Span\br{H_n} = 4c(H_n) + 2(\chi(H_n) -2)$. 
Since $c(H_n)=n$ and we can see that $\sharp H_n(s_A)=\sharp H_n(s_B)=1$, 
Lemma~\ref{Atom_LemEulercha} tells us that $\chi(H_n)=2-n$ and hence $\Span\br{H_n}=2n$.
In particular, if $n$ is odd, then  $\Span\br{H_n}$ cannot be divided by $4$ and therefore $H_n$ cannot represent any classical links. 
\end{exam}

The rest of this paper is devoted to the proofs of Theorems \ref{Atom_SpanPseudoAdequate} and \ref{Atom_SpanVadequate}.

For arbitrarily chosen real crossings $p_1,\dots,p_j$ ($1\le j\le c(D)$), 
let $s_A(j)$ (resp. $s_B(j)$) be a state of $D$ 
that maps $p_1,\dots,p_j$ to B (resp.\ A) and the other real crossings to A (resp.\ B). 
Note that any $s\in\mathcal{S}$ other than $s_A$ and $s_B$ can be expressed as $s_A(j)$ or $s_B(j)$. 
By definition
	\begin{align}
	\odeg\br{D/s_A(j)} & = c(D) -2j + 2\sharp D(s_A(j)) -2 , \label{MurasugiSpan_ThOdegDsAj}\\
	\udeg\br{D/s_B(j)} & = -c(D) +2j - 2\sharp D(s_B(j)) +2. \label{MurasugiSpan_ThOdegDsBj}
	\end{align}
%
\begin{lem} \label{odegsA(j)<sAudegsB(j)<sB}
If $j \ge 1$, then we have 
	\[
	\odeg\br{D/s_A(j)} \le \odeg\br{D/s_A},\quad 
	\udeg\br{D/s_B(j)} \ge \udeg\br{D/s_B}.
	\]
Thus $\Span\br{D} \le \odeg\br{D/s_A} - \udeg\br{D/s_B}$.
\end{lem}
\begin{proof}
Figure \ref{image_sA(j-1)sA(j)} shows 
	\begin{align}
	\label{MurasugiSpan_claim_sA(k-1)sA(k)}
	\sharp D(s_A(j-1))-1 \le \sharp D(s_A(j)) \le \sharp D(s_A(j-1))+1.
	\end{align}
	\begin{figure}
	\centering
	\input{sA_j-1_sA_j}
	\caption{Proof of \eqref{MurasugiSpan_claim_sA(k-1)sA(k)}}
	\label{image_sA(j-1)sA(j)}
	\end{figure}
By \eqref{MurasugiSpan_claim_sA(k-1)sA(k)} we inductively deduce 
	\begin{equation}\label{sharpD(sA(j)<sharpD(sA)+j}
	 \sharp D(s_A(j)) 
	 \le \sharp D(s_A(j-1))+1 
	 \le \sharp D(s_A(j-2))+2 
	 \le\dots\le
	 \sharp D(s_A) + j. 
	\end{equation}
Similarly, we have 
	\begin{equation}
	\sharp D(s_B(j)) \le \sharp D(s_B) + j. \label{sharpD(sB(j)<sharpD(sB)+j}
	\end{equation}
By definition
	\begin{align}
	\odeg\br{D/s_A} & =  c(D) + 2\sharp D(s_A) -2, \label{MurasugiSpan_ThOdegDsA}\\
	\udeg\br{D/s_B} & = -c(D) - 2\sharp D(s_B) +2. \label{MurasugiSpan_ThOdegDsB}
	\end{align}
By \eqref{sharpD(sA(j)<sharpD(sA)+j} and \eqref{sharpD(sB(j)<sharpD(sB)+j}, 
	\begin{align*}
	\odeg\br{D/s_A(j)} 
	& = c(D) - 2j + 2\sharp D(s_A(j)) -2 \\ 
	& \le c(D) + 2\sharp D(s_A) -2, \nonumber \\
	\udeg\br{D/s_B(j)} 
	& = -c(D) + 2j - 2\sharp D(s_B(j)) +2 \\ 
	& \ge -c(D) - 2\sharp D(s_B) +2 \nonumber.
	\end{align*}
These inequalities and \eqref{MurasugiSpan_ThOdegDsA}, \eqref{MurasugiSpan_ThOdegDsB} imply 
\[
 \odeg\br{D} \le \odeg\br{D/s_A}
 \quad\text{and}\quad
 \udeg\br{D} \ge \udeg\br{D/s_B}.
\]
Thus we have $\Span\br{D} \le \odeg\br{D/s_A} - \udeg\br{D/s_B}$.
\end{proof} 
%
%
\begin{proof}[Proof of Theorem \ref{Atom_ThAtomProperSpan}]
Lemma \ref{Atom_LemEulercha} and Lemma \ref{odegsA(j)<sAudegsB(j)<sB} imply
\begin{align}\label{AtomSpanInequality}
    \Span\br{D}
    & \le \odeg\br{D/s_A} - \udeg\br{D/s_B}  \\
    & = 2c(D) + 2(\sharp D(s_A) + \sharp D(s_B)) -4 \notag  \\
    & = 2c(D) + 2(\chi(D) + c(D)) -4 \notag\\
    & = 4c(D) + 2(\chi(D) -2) .\qedhere
\end{align}
\end{proof}
%
%
%
\begin{proof}[Proof of Theorem~\ref{Atom_SpanPseudoAdequate}]
Pseudo-adequacy of $D$ implies
	\begin{equation}\label{eq:D(s(1))=<D(s)}
	 \sharp D(s_A(1)) \le \sharp D(s_A),
	 \quad
	 \sharp D(s_B(1)) \le \sharp D(s_B).
	\end{equation}
Thus \eqref{sharpD(sA(j)<sharpD(sA)+j} and \eqref{sharpD(sB(j)<sharpD(sB)+j} can be sharpend in this case as 
	\begin{align} \label{Atom_claimApropersAsB}
	\sharp D(s_A(j)) \le \sharp D(s_A) + j-1, \quad \sharp D(s_B(j)) \le \sharp D(s_B) + j-1.
	\end{align}
\eqref{MurasugiSpan_ThOdegDsA}, \eqref{MurasugiSpan_ThOdegDsB} and \eqref{Atom_claimApropersAsB} imply
	\begin{align}
	\odeg\br{D/s_A(j)} 
	& = c(D) - 2j + 2\sharp D(s_A(j)) -2 \nonumber \\
	& \le c(D) + 2\sharp D(s_A) -4, \label{Atom_3_ThAproperOdegDsAj} \\
	\udeg\br{D/s_B(j)} 
	& = -c(D) + 2j - 2\sharp D(s_B(j)) +2 \nonumber \\
	& \ge -c(D) - 2\sharp D(s_B) +4.\label{Atom_3_ThAproperOdegDsBj}
	\end{align}
These estimations deduce $\odeg\br{D/s_A(j)}<\odeg\br{D/s_A}$ and $\udeg\br{D/s_B(j)}>\udeg\br{D/s_B}$ for any choices of $p_1,\dots,p_j$ ($j\ge 1$), and hence 
we have $\odeg\br{D} = \odeg\br{D/s_A}$ and $\udeg\br{D} = \udeg\br{D/s_B}$. 
Thus the inequality in \eqref{AtomSpanInequality} becomes an equality. 
\end{proof}

	\begin{proof}[Proof of Theorem \ref{Atom_SpanVadequate}]
In general, 
it is not hard to see that a connecting arc in $D(s)$ ($s\in\mathcal{S}$) corresponing to a real crossing $p$ is admissible 
if and only if $\sharp D(s(1)) \le \sharp D(s)$
where $s(1)$ is the state of $D$ obtained from $s$ 
by changing splice at $p$. 
Thus if all the connecting arcs of $D(s_A)$ and those of $D(s_B)$ are admissible, then $D$ is pseudo-adequate.

If $D$ is such a diagram as in Theorem \ref{Atom_SpanVadequate}, all the connecting arcs of $D(s_A)$ and $D(s_B)$ are admissible;
because $D'$ is (pseudo-)adequate,
the connecting arcs one of whose endpoints is not on $C_A$ nor $C_B$ are admissible.
By assumption the other arcs are also admissible. 
\end{proof}
%
%
\section*{Acknowledgements}
The authors express their sincere appreciation to Professor Naoko Kamada and Professor Shin Satoh for their fruitful comments and encouragements.
KS is partially supported by JSPS KAKENHI Grant Number 16K05144.

\begin{thebibliography}{99}
\bibitem{BaeMorton03}
Y.~Bae, H.~R.~Morton,
\emph{The spread and extreme terms of Jones polynomials},
J.\ Knot Theory Ramifications \textbf{12} (3) (2003), 359-373

\bibitem{DFKLS08}
O.~T.~Dasbach, D.~Futer, E.~Kalfagianni, X.-S.~Lin, N.~Stoltzfus,
\emph{The Jones polynomial and graphs on surfaces},
J.\ Combin.\ Theory Ser.\ B \textbf{98} (2) (2008), 384-399

\bibitem{Kamada04}
N.~Kamada, 
\emph{Span of the Jones polynomial of an alternating virtual link}, 
Algebr.\ Geom.\ Topol.\ \textbf{4} (2004), 1083-1101

\bibitem{Kauffman87}
L.~Kauffman,
\emph{State models and the Jones polynomial},
Topology \textbf{26} (1987), 395-407

\bibitem{Kauffman99}
L.~Kauffman, 
\emph{Virtual knot theory}, 
European J. Combin. 20 (1999), 663-690

\bibitem{Kishino00}
T.~Kishino,
\emph{On classification of virtual links whose crossing numbers are equal to or less than 6} (in Japanese),
Master Thesis, Osaka City University (2000)

\bibitem{LickorishThistlethwaite88}
W.~B.~R. Lickorish and M.~B.~Thistlethwaite,
\emph{Some links with non-trivial polynomials and their crossing numbers},
Comment.\ Math.\ Helv.\ \textbf{63} (1988), 527-539

\bibitem{ManturovIlyutko10}
V.~O.~Manturov and D.~P.~Ilyutko,
\emph{Virtual Knots. The state of the Art}.
Series on Knots and Everything, 51,
World Scientific Publishing Co. Pte. Ltd., Hackensack, NJ (2013)

\bibitem{Murasugi87}
K.~Murasugi,
\emph{Jones polynomials and classical conjectures in knot theory},
Topology \textbf{26} (1987), 187-194

\bibitem{Thistlethwaite87}
M.~B.~Thistlethwaite,
\emph{A spanning tree expansion of the Jones polynomial},
Topology \textbf{26} (1987), no.\ 3, 297-309

\bibitem{Turaev87}
V.~G.~Turaev,
\emph{A simple proof of the Murasugi and Kauffman theorems on alternating links},
Enseign.\ Math.\ (2) \textbf{33} (1987), no.\ 3-4, 203-225

\end{thebibliography}
\end{document}